\documentclass[11pt,twoside, leqno]{article}

\usepackage{amsthm,color}
\usepackage{amssymb}
\usepackage{amsmath}
\usepackage{mathrsfs}
\usepackage{txfonts}
\usepackage{graphics}
\usepackage[Symbol]{upgreek}

\allowdisplaybreaks 
\newcommand{\ls}{\leqslant}
\newcommand{\gs}{\geqslant}

\newcommand{\ip}[2]{\langle{#1},{#2}\rangle}

\def\rr{{\mathbb R}}

\def\supp{{\mathop\mathrm{\,supp\,}}}

\def\loc{{\mathop\mathrm{\,loc\,}}}
\def\Lip{{\mathop\mathrm{\,Lip\,}}}

\def\ez{\epsilon}

\def\gz{{\gamma}}

\def\sz{\sigma}

\def\r{\right}
\def\lf{\left}

\def\la{\langle}
\def\ra{\rangle}

\newtheorem{thm}{Theorem}[section]
\newtheorem{lem}{Lemma}[section]

\newtheorem{rem}{Remark}[section]

\newtheorem{defn}{Definition}[section]

\newtheorem{ques}{Question}[section]

\numberwithin{equation}{section}

\begin{document}
\arraycolsep=1pt

\title{\bf\Large Hamilton's Gradient Estimates and A Monotonicity Formula \\ for Heat Flows on Metric Measure Spaces
\footnotetext{\hspace{-0.35cm}
  2010 \it{Mathematics Subject Classification}. 53C23, 35K05
\endgraf
{\it Key words and phrases}. Metric measure space, Curvature-dimension condition, Harmilton's gradient estimates, Entropy.
\endgraf
R.Jiang is partially supported by NSFC (No. 11301029);  H.Zhang is partially supported by NSFC (No. 11201492) and the Fundamental Research Funds for the Central Universities.
}}
\author{Renjin Jiang \& Huichun Zhang}
\date{}

\maketitle

\begin{center}
\begin{minipage}{11cm}\small
{\noindent{\bf Abstract}.   In this paper, we extend the
Hamilton's gradient estimates \cite{har93} and a monotonicity formula
of entropy \cite{ni04} for heat flows from smooth Riemannian
manifolds to (non-smooth) metric measure spaces with appropriate
Riemannian curvature-dimension condition.
}\end{minipage}
\end{center}

\section{Introduction}
\hskip\parindent
In their fundamental work \cite{ly86}, Li and Yau discovered a gradient estimate for positive solutions of the heat equation
\begin{equation}\label{eq1.1}
\partial_tu=\Delta u
\end{equation}
on a smooth Riemannian manifold with Ricci curvature bounded from below. Later on,  Hamilton in \cite{har93} used the similar method to establish the following gradient estimates.
\begin{thm}[Hamiton \cite{har93}] Let $(M,g)$ be an $n$-dimensional compact Riemannian
manifold with Ricci curvature $R_{ij}\gs -Kg_{ij}$ for some $K\gs 0$ and
$\partial M=\varnothing$. If $u(x,t)$ is a positive solution of the heat equation with $0<u\ls M$. then
\begin{equation*}
t|\nabla \log u|^2\ls (1+2Kt)\cdot\log(M/u).
\end{equation*}
\end{thm}
\noindent Such type gradient estimates are called as Li-Yau-Hamilton (LYH, for short) inequalities afterwards.
Recently, Kotschwar \cite{ko07} extended
Theorem 1.1 to  complete noncompact manifolds.

The LYH inequality is of the basic tools, that has been  widely used in geometric analysis.
For instance, the LYH inequalities imply the classical Harnack inequalities for heat equations, by integrating the LYH type gradient estimates along space-time curves.
 The research of Li-Yau-Hamilton estimates for heat equation
 (or other geometric evolution equations) on smooth Riemannian manifolds
 has a long history. For an overview, the reader is referred to Chapter 4 in
 book \cite{sy94} and \cite{sz06,lx11,wu13,q12,bbg14,bq00}, and references therein. Very recently,
 Bakry-Bolley-Gentil \cite{bbg14} have established an optimal global Li-Yau
on smooth Markov semi-group under a curvature-dimension condition.

Li-Yau's gradient estimates have been extended from smooth manifolds to compact
Alexandrov spaces in \cite{qzz13}, and recently to general metric measure spaces
with the \emph{Riemannian curvature-dimension condition $RCD^\ast(K,N)$}, by
Garofalo-Mondino \cite{gam14} (for the case
$\mu(X)<\infty$) and \cite{jiang14-2} (for the case $\mu(X)=\infty$). Gradient estimates for harmonic functions on metric measure spaces have been studied in \cite{hkx13,jiang14}.
We refer the readers to \cite{agmr,ags1,ags3,ams13,ams13-1,eks13} for recent developments for {Riemannian} curvature
dimension conditions $RCD^\ast(K,N)$,  and  \cite{lv09,stm4,stm5} for curvature
dimension conditions $CD(K,N)$, on metric measure spaces; see Section 2 below.

One of our main aim of this paper is to consider the LYH inequality for heat equations on non-smooth metric measure spaces $(X,d,\mu)$.
 Precisely, our first main result is the following:
\begin{thm}\label{main}
 Let $(X,d,\mu)$ be a proper metric measure space satisfying $RCD^*(K,\infty)$, where $K\ls 0$. Let $u(x,t)$ be a positive
 solution of the heat equation on $X\times [0,\infty)$ with initial value $u(x,0)=u_0(x)\in \cup_{1\ls q<\infty} L^q(X)$. Suppose that there exists a positive constant $M$ such that $u_0(x)\ls M$  for almost every $x\in X$.
 Then
\begin{equation}\label{eq1.2}
t|\nabla \log u|^2\ls (1-2Kt)\cdot\log(M/u)
\end{equation}
 for almost every $(x,t)$ in $X\times [0,\infty)$.
 \end{thm}

 On the other hand, inspired by Perelman's $\mathcal W$-entropy, Ni \cite{ni04} introduced an entropy for the heat equation on an $n$-dimensional smooth Riemannian manifold $(M,g)$. For any smooth function $f$ on $M$ and $\tau>0$ with $\int_M(4\pi\tau)^{-n/2}e^{-f}dx=1,$ the entropy is given as
 $$\mathcal W(f,\tau):=\int_M\Big(\tau|\nabla f|^2+f-n\Big)\frac{e^{-f}}{(4\pi\tau)^{n/2}}dx.$$
 Let $u(x,t)$ be a positive solution of (\ref{eq1.1}) on a closed (i.e., compact and without boundary)
 manifold $M$ with $\int_Mudx=1$. Let $f$ be defined as $u=(4\pi\tau)^{-n/2}e^{-f}$ and
 $\tau=\tau(t)$ with $\frac{d\tau}{d t}=1$. In \cite{ni04}, Ni proved that, if $M$ has
 nonnegative Ricci curvature, then the entropy $\mathcal W(f,\tau)$ is monotone decreasing along
 the heat equation (as $t\to\infty$). Such monotonicity also discussed in \cite{c12}. Wang \cite{w13} extended it to a compact manifold with Bakry-\'Emery curvature bounded below.

 Our second main result is to extend Ni's monotonicity for heat equation to non-smooth metric measure space $(X,d,\mu)$.
Precisely, we have the following result.
\begin{thm}\label{thm1.3}
 Let $(X,d,\mu)$ be a {compact} metric measure space satisfying $RCD^*(0,N)$ with $N\in[1,\infty]$. Let $u(x,t)$ be a positive
 solution of the heat equation on $X\times [0,\infty)$ with initial value $u(x,0)=u_0(x)\in L^\infty(X)$.
  Then we have the following:\\
\indent  $(i)$ If $N=\infty$, by letting $f:=-\log u$, then the entropy
$$\mathcal W_\infty(f,t):=\int_X|\nabla f|^2\cdot e^{-f}d\mu  $$
 is monotone decreasing (as $t\to\infty$).\\
\indent  $(ii)$ If $N<\infty$, by letting  $f$ be defined as $u=(4\pi\tau)^{-N/2}e^{-f}$ and $\tau=\tau(t)>0$ with $\frac{d\tau}{d t}=1$, then the entropy
 $$\mathcal W_N(f,\tau):=\int_M\Big(\tau|\nabla f|^2+f-N\Big)\frac{e^{-f}}{(4\pi\tau)^{N/2}}d\mu$$
 is monotone decreasing (as $t\to\infty$).
\end{thm}
According to the referee's suggestion, we post the following question:
\begin{ques}
Extend Theorem 1.3 to the case when  $(X,d,\mu)$ is not compact and $\mu$ is a $\sigma$-finite measure on $X$.
\end{ques}

The paper is organized as follows. In Section 2, we recall some necessary materials for Sobolev spaces and
Riemannian curvature-dimension condition on metric measure spaces.
 Hamilton's gradient estimate \eqref{eq1.2}  will be proved in the third section. In the last section, we will prove the monotonicity result Theorem \ref{thm1.3}.\\

\noindent{\textbf{Acknowledgments}}. We are grateful to the anonymous referee for his/her several
interesting and helpful suggestions and comments, which improves this paper.

\section{Preliminaries}
\hskip\parindent In this section, we recall some basic notions and several auxiliary results. Let $(X,d)$ be a complete, separable and proper metric space,
$\mu$ be a $\sigma$-finite Radon measure, with ${\rm supp}(\mu)=X.$

\subsection{Sobolev spaces on metric measure spaces}
\hskip\parindent Let $C([0, 1],X)$ be the space of continuous curves on $[0, 1]$ with values in $X$, which we
endow with the sup norm. For $t\in [0, 1]$, the map $e_t :C([0, 1],X) \to X$ is the evaluation at
time $t$ defined by
$$e_t(\gz):=\gz_t.$$
Given a non-trivial closed interval $I\subset \rr$, a curve $\gz: I \to X$ is in the absolutely continuous
class $AC^q([0,1],X)$ for some $q\in [1,\infty]$, if there exists $f\in L^q(I)$ such that, for all $s,t\in I$ and $s<t$, it holds
$$d(\gz_t,\gz_s)\le\int_s^tf(r)\,dr.$$

\begin{defn}[Test Plan]
Let $\uppi\in {\mathcal{P}}(C([0, 1],X))$.
We say that $\uppi$ is a test plan if there exists $C>0$ such that
$$(e_t)_\sharp{ \uppi}\le C\mu, \ \forall\,t\in [0,1],$$
and
$$\int \int_0^1|\dot{\gamma}_t|^2\,dt\,d{\mathcal \uppi}(\gamma)<\infty.$$
  \end{defn}

\begin{defn}[Sobolev Space] \label{sobolev} The Sobolev class $S^2(X)$ (resp. $S_{\mathrm{loc}}^2(X)$)
is the space of all Borel functions $f: X\to \rr$, for which there exists a non-negative function
$G\in L^2(X)$ (resp. $G\in L^2_{\mathrm{loc}}(X)$) such that, for each test plan $\uppi$, it holds
\begin{equation}\label{curve-sobolev}
\int |f(\gz_1)-f(\gz_0)|\,d{\mathcal \uppi}(\gamma)\le \int \int_0^1 G(\gz_t)|\dot{\gamma}_t|^2\,dt\,d{\mathcal \uppi}(\gamma).
\end{equation}
\end{defn}

It has been proved that (see \cite{ch,sh,ags4}), for each  $f\in S^2(X)$, there exists a unique minimal $G$
in the $\mu$-a.e. sense such that \eqref{curve-sobolev} holds. We then denote the minimal $G$ by $|\nabla  f|$
and call it the minimal weak upper gradient following \cite{ags4}.

We then define the in-homogeneous Sobolev space $W^{1,2}(X)$ as  $S^2(X)\cap L^2(X)$ equipped with the norm
$$\|f\|_{W^{1,2}(X)}:=\lf(\|f\|_{L^2}^2+\| |\nabla  f|\|_{L^2}^2\r)^{1/2}.$$

\begin{defn}[Local Sobolev Space]  Let $\Omega\subset X$ be an open set. A Borel function $f: \Omega\to \rr$ belongs to
$S^2_{\mathrm{loc}}(\Omega)$, provided, for any Lipschitz function $\chi: X \to\rr$ with $\supp(\chi)\subset \Omega$,  it holds
$f\chi\in S^2_{\mathrm{loc}}(X)$. In this case, the function $|\nabla f| : \Omega\to[0,\infty]$ is $\mu$-a.e. defined by
$$|\nabla f| := |\nabla (\chi f)|,\ \  \mu-a.e. \,\mathrm{on} \ \{\chi=1\},$$
 for any $\chi$ as above. The space $S^2(\Omega)$ is the collection of such $f$ with $|\nabla f|\in L^2(\Omega)$.
\end{defn}

 Notice that, if $\mu$ is locally doubling and $(X,d,\mu)$ supports a local weak $L^2$-Poincar\'e inequality, the Sobolev space $W^{1,2}(X)$
  coincides with the Sobolev spaces based on upper gradients introduced by Cheeger \cite{ch} and Shanmugalingam \cite{sh};
  see Ambrosio, Gigli and Savar\'e \cite{ags4}.

Fixed any open set $\Omega\subset X$. We denote by $\Lip(\Omega)$ (or $\Lip_{\rm loc}(\Omega)$, or $\Lip_0(\Omega)$) the space of Lipschitz continuous functions on $\Omega$ (resp. or locally Lipschitz continuous functions, or Lipschitz continuous functions with compact support on $\Omega$).
The local Sobolev space $W^{1,2}_{\mathrm{loc}}(\Omega):=L^2_{\rm loc}(\Omega)\cap S^2_{\mathrm{loc}}(\Omega)$,  and the Sobolev space with compact support
$W^{1,2}_{0}(\Omega)$ is defined as the completion of $\Lip_0(\Omega)$ with respect to the $W^{1,2}$-norm. We also denote $C_0(\Omega)$ by the continuous function with compact support on $\Omega$ with uniform norm.

\subsection{Differential structure and the Laplacian}
\hskip\parindent  The following terminologies and results are mainly taken from \cite{ags3,gi12}.
\begin{defn}[Infinitesimally Hilbertian Space] Let $(X, d,\mu)$ be a metric measure
space. We say that it is infinitesimally Hilbertian, provided $W^{1,2}(X)$ is a Hilbert space.
\end{defn}
Notice that, from the definition, it follows that $(X, d,\mu)$ is infinitesimally Hilbertian if and only if,
for any $f,g\in S^2(X)$, it holds
$$\||\nabla (f+g)|\|_{L^2(X)}^2+\||\nabla (f-g)|\|_{L^2(X)}^2=2\lf(\||\nabla f|\|_{L^2(X)}^2+\||\nabla g|\|_{L^2(X)}^2\r).$$

\begin{defn} Let $(X, d,\mu)$ be an infinitesimally Hilbertian space, $\Omega\subset X$ an open set and $f, g\in S^2_{\mathrm{loc}}(\Omega)$.
The map $\la \nabla f, \nabla g\ra :\, \Omega \to \rr$ is $\mu$-a.e. defined as
$$\la \nabla f, \nabla g\ra:= \inf_{\ez>0} \frac{|\nabla (g+\ez f)|^2-|\nabla g|^2}{2\ez}$$
the infimum being intended as $\mu$-essential infimum.
\end{defn}

The inner product $\la \nabla f, \nabla g\ra$ is linear, and satisfies Cauchy-Schwartz inequality,
Chain rule and Leibniz rule; see Gigli \cite{gi12}.

The inner product also provides a canonical strongly local Dirichlet form $(\mathscr E,W^{1,2}(X))$  by
$$\mathscr{E}(f,g):=\int_X\la \nabla f,\nabla g\ra\,d\mu,\qquad\forall\ f,g\in W^{1,2}(X).$$
We denote by $\Delta_{\mathscr E}$ and  $H_t$  the generator  of $\mathscr E$ and the corresponding  semigroup (heat flow) $e^{t\Delta_{\mathscr E}}$.
The domain of the generator is denoted by ${\mathscr D}(\Delta_{\mathscr E})$. For each $u_0\in L^2(X)$, the heat flow $H_tu_0(x)$ provides a (unique) solution of heat equation $\partial_t u=\Delta_{\mathscr E} u$ in $W^{1,2}(X).$

The generator $\Delta_{\mathscr E}$ gives a natural definition of Laplacian. However, with the aid of the inner product, Gigli in \cite{gi12} and Gigli-Mondino in\cite{gm13} have defined a measure-valued Laplacian operator (or more general elliptic operators) as below (see also \cite{per11,zz12} for the case of Alexandrov spaces).

\begin{defn}[Laplacian] \label{glap}
Let $(X, d,\mu)$ be an infinitesimally Hilbertian space.
Let $f\in W^{1,2}_{\mathrm{loc}}(X)$. We call $f\in {\mathscr D}(\Delta^*)$, if there exists a signed Radon measure $\nu$  such that, for each
 $\psi\in \Lip_0(X)$, it holds
\begin{equation*}
-\int_X\la \nabla f,\nabla\psi\ra\,d\mu=\int_X \psi\,d\nu.
\end{equation*}
 If such $\nu$ exists, it must be unique. We will write $\Delta^* f=\nu$.

 We also denote by
  $${\mathscr D}(\Delta):=\left\{f\in {\mathscr D}(\Delta^*)\ :\ \Delta^* f\ll \mu \ \ {\rm and\ the\ density}\ \ \frac{d(\Delta^*f)}{d\mu}\in L_{\loc}^2(X)\right\}$$
  When $f\in {\mathscr D}(\Delta)$, we denote $\Delta f:=\frac{d(\Delta^*f)}{d\mu}$.
\end{defn}
\noindent Note that, if $f\in \mathscr D(\Delta^*)$, the test function $\psi$ can be chosen in $W^{1,2}_0(X)\cap L^1_{\rm loc}(X,|\nu|).$ The operator $\Delta^*$ is linear
due to $(X, d,\mu)$ being infinitesimally Hilbertian.

 It is clear that  ${\mathscr D}(\Delta_{\mathscr E})\subset{\mathscr D}(\Delta)\ (\varsubsetneq{\mathscr D}(\Delta^*)).$ Since we consider only that $X$ is proper,
it is proved in \cite[Proposition 4.24]{gi12} that
$${\mathscr D}(\Delta)\cap W^{1,2}(X)={\mathscr D}(\Delta_{\mathscr E})\qquad {\rm and}\qquad
  \Delta f=\Delta_{\mathscr E}f \quad\forall  f\in {\mathscr D}(\Delta_{\mathscr E}).$$

Such a measure-valued Laplacian $\Delta^*$ satisfies the following Chain rule and Leibniz rule (we consider only the case where $X$ is proper):

\begin{lem}\label{lem2.1}  Let $(X,d,\mu)$ be infinitesimally Hilbertian and proper.\\
{\rm(i)\ (Chain rule)}\ \ \  Let $g\in {\mathscr D}(\Delta^*)\cap L^\infty_{\rm loc}(X)$ and $\phi\in C^{1,1}_{\rm loc}(\rr)$. If $g\in C(X)$ or
$g\in {\mathscr D}(\Delta)$, then we have
\begin{equation*}
\phi\circ g\in{\mathscr D}(\Delta^*)\qquad {\rm and}\qquad\Delta^*(\phi\circ g)=\phi'\circ g\cdot\Delta^*g+\phi''\circ g|\nabla g|^2\cdot\mu.
\end{equation*}
{\rm(ii)\ (Leibniz rule)}\ \ \ Let $g_1,g_2\in {\mathscr D}(\Delta^*)$. Then we have
\begin{equation*}
g_1\cdot g_2\in {\mathscr D}(\Delta^*)\qquad {\rm and}\qquad \Delta^*(g_1\cdot g_2)=g_1\cdot\Delta^*g_2+g_2\cdot\Delta^*g_1+2\ip{\nabla g_1}{\nabla g_2}\cdot\mu
\end{equation*}
provided that $g$ satisfies one of the following three conditions:\\
\indent \rm{(a)}\ \ $g_1,g_2\in C(X)$;\\
\indent  \rm{(b)}\ \ $g_1,g_2\in {\mathscr D}(\Delta)\cap L^\infty_{\rm loc}(X)$;\\
\indent  \rm{(c)}\ \ $g_2\in W^{1,2}(X)$ and $g_1\in {\rm Lip}_{\rm loc}(X)\cap{\mathscr D}(\Delta)$.
\end{lem}
\begin{proof}
Chain rule (i) and (a), (b) in Leibniz rule (ii) have proved in \cite[Proposition 4.28 and Proposition 4.29]{gi12} by Gigli. We need only to check the condition (c) in Leibniz rule.

 We argue by  approximation. Since $g_2\in W^{1,2}(X)$, it is shown in \cite{ags3} that there exists a sequence $\{\widetilde{h}_j\}_{j=1}^\infty\subset {\rm Lip}(X)\cap W^{1,2}(X)$ such that $$\widetilde{h}_j\overset{L^2}{\to}g_2\qquad {\rm and}\qquad |\nabla \widetilde{h}_j|\overset{L^2}{\to}|\nabla  g_2|$$ as $j\to\infty$. Notice that $W^{1,2}(X)$ is a Hilbert space (hence it is reflexive), we can use Mazur's lemma to conclude that there exists a convex combination of $\{\widetilde{h}_j\}$, denoted by $\{h_j\}$, which strongly converges  to $g_2$ in $W^{1,2}(X).$ On the other hand, $\mathscr D(\Delta_{\mathscr E})$ is also dense in $W^{1,2}(X)$. Thus, without loss of generality, we can assume that $h_j\in {\rm Lip}(X)\cap \mathscr D(\Delta_{\mathscr E})$ for all $j$, and  $h_j\overset {W^{1,2}(X)}{\to}g_2$.

 Since $g_1\in {\rm Lip}_{\rm loc}(X)\cap{\mathscr D}(\Delta)$ and $h_j\in {\rm Lip}(X)\cap \mathscr D(\Delta_{\mathscr E})$, by (b), we have, for each $j\in \mathbb N$,
 $$\Delta^*(g_1\cdot h_j)=g_1\cdot\Delta^*h_j+h_j\cdot\Delta^* g_1+2\ip{\nabla g_1}{\nabla h_j}\cdot\mu.$$
  Let us fix a test function $\phi\in \Lip_0(X)$. We have
    $$-\!\int_X\!\ip{\nabla \phi}{\nabla (g_1  h_j)}d\mu=-\!\int_X\!\ip{\nabla (\phi g_1)}{\nabla  h_j}d\mu -\!\int_X\!\ip{\nabla (h_j \phi)}{\nabla g_1}d\mu +2\!\int_X\!\phi\cdot\ip{\nabla g_1}{\nabla  h_j}  d\mu.$$
 Notice that the combination of $g_1\in {{\rm Lip}_{\rm loc}(X)} $ and $h_j\overset {W^{1,2}(X)}{\to}g_2$ implies that $g_1\cdot h_j\overset {W_{\rm loc}^{1,2}(X)}{\to}g_1\cdot g_2$.
 The same holds for the sequence  $\{\phi \cdot h_j\}$.
 Letting $j\to\infty$, we can get
 $$-\!\int_X\!\ip{\nabla \phi}{\nabla (g_1g_2)}d\mu=-\!\int_X\!\ip{\nabla(\phi g_1) }{\nabla g_2}d\mu-\!\int_X\!\ip{\nabla(\phi g_2)}{\nabla g_1}d\mu+2\int_X\phi\cdot\ip{\nabla g_1}{\nabla
 g_2}d\mu.$$
 This proves the Lemma.
\end{proof}

\subsection{Curvature-dimension conditions and consequences}
\hskip\parindent
We shall use the following definition for $RCD^*(K,N)$ spaces, which is equivalent to the original definition \cite{eks13,ams13-1}.
Here and in the sequel, for $K = 0$, $ \frac{2Kt}{e^{2Kt}-1}:= \lim_{K\to0} \frac{2Kt}{e^{2Kt}-1} = 1$.
\begin{defn}[$RCD^*(K,N)$ Space]\label{rcd}
Let $(X,d,\mu)$ be  an infinitesimally Hilbertian space, and let $H_t$ be the semigroup corresponding to the previous Dirichlet form $(\mathscr E,W^{1,2}(X))$.

Given $K\in \rr$ and $N\in [1,\infty]$, the space $(X,d,\mu)$ is called a $RCD^*(K,N)$ space,
 if the following three conditions are satisfied:\\
 {\rm (i)}\ there exist $x_0\in X$, and constants $c_1,c_2>0$, such that
\begin{equation}\label{vol-growth}
\mu(B(x_0,r))\le c_1\cdot e^{c_2r^2};
\end{equation}
{\rm (ii)}  for all $f\in W^{1,2}(X)$ and each $t>0$, it holds for $\mu$-a.e. $x\in X$ that
\begin{equation}\label{eq2.3}
|\nabla H_tf(x)|^2\le e^{-2Kt}H_t(|\nabla f|^2)(x)-\frac{4Kt^2}{N(e^{2Kt}-1)}|\Delta H_tf(x)|^2,
\end{equation}
where, if $N=\infty$, the last term is understood as $0$;\\
{\rm (iii)} if $f\in W^{1,2}(X)\cap L^\infty(X)$ satisfying $|\nabla f|\ls 1$, then $f$ has an 1-Lipschitz representative.
\end{defn}

Let $(X,d,\mu)$ be a $RCD^*(K,N)$ space with $K\in\mathbb R$ and $N\in[1,\infty)$. Then the measure $\mu$ is local doubling, and hence $X$ is proper.

It is known that, (see \cite[\S 4]{gi12}), for each $t>0$, the operator $H_t$ is bounded
from $L^p(X)$ to $L^p(X)$ for any $p\in[1,\infty)$. Thus, the semigroup
(heat flow) can be extended on $L^p(X)$ for any $p\in [1,\infty).$ Therefore, for each $p\in [1,\infty)$  and $u_0(x)\in L^p(X)$, the heat flow $H_tu_0(x)$ provides a solution of heat equation $\partial_t u=\Delta u$. When $p\in (1,\infty)$, the $L^p$ solution of the heat equation is uniquely determined by its initial value in $L^p(X)$ (see \cite{str83,li84}).

\begin{lem}[\cite{agmr,ags1}]\label{lem2.2}\indent
Let $K\in \rr$, and let $(X,d,\mu)$ be  a proper $RCD^*(K,\infty)$ space. Then,  for any $t>0$ and $f\in L^\infty(X)\cap L^2(X)$, we have $H_tf\in {\rm Lip}(X)$ and
$$ {\rm Lip}(H_tf)\ls \frac{\|f\|_{L^\infty}}{\sqrt{2\int_0^te^{2Ks}ds}}.$$
\end{lem}
Our proof of Theorem \ref{main} and Theorem \ref{thm1.3} relies on some self-improvements of regularity for heat flows under the Riemannian
curvature-dimension condition.
\begin{lem}[\cite{sa14}]\label{lem2.3}\indent
Let $K\in \rr$ and let $(X,d,\mu)$ be  a proper $RCD^*(K,\infty)$ space.
If $ f\in {\mathscr D}(\Delta_{\mathscr E})\cap \Lip(X)\cap L^\infty(X)$ with $ \Delta f\in W^{1,2}(X).$ Then $|\nabla f|^2\in W^{1,2}(X)\cap L^\infty(X)$ and $|\nabla f|^2\in {\mathscr D}(\Delta^*)$.
When we write the Lebesgue's decomposition of $\Delta^*(|\nabla f|^2)$ w.r.t $\mu$ as
 $$\Delta^*(|\nabla f|^2)=\Delta^{R}(|\nabla f|^2) \cdot\mu+\Delta^{S}(|\nabla f|)^2,$$
 then we have the estimates that $\Delta^{S}(|\nabla f|)^2\gs0$ and, for $\mu$-a.e. $\ x\in X$,
\begin{equation}\label{eq2.4}
\frac{1}{2} \Delta^{R}(|\nabla f|^2)\gs \frac{1}{N}(\Delta f)^2+ \la\nabla\Delta f,\nabla f \ra+K|\nabla f|^2.
\end{equation}
Furthermore,  we have, for  $\mu$-a.e. $  x\in \big\{y:\ |\nabla f(y)|\not=0\big\}$,
\begin{equation}\label{eq2.5}
\frac{1}{2} \Delta^{R}(|\nabla f|^2)\gs \frac{1}{N}(\Delta f)^2+ \la\nabla\Delta f,\nabla f \ra+K|\nabla f|^2+\frac{N}{N-1}\cdot\Big(\frac{\la \nabla f,\nabla |\nabla f|^2\ra}{2|\nabla f|^2}-\frac{\Delta f}{N}\Big)^2.
\end{equation}
Here and in the sequel, if $N=\infty$, then $\frac{1}{N}=0$ and $\frac{N}{N-1}=1.$
\end{lem}
\begin{proof} In the case of $N=\infty$, the lemma is Theorem 3.4 in \cite{sa14}.

In the case of $N<\infty$,
the same argument in the proof of Lemma 3.2 in \cite{sa14} implies (\ref{eq2.4}) (See also
Theorem 2.7 of \cite{gam14}). The last inequality (\ref{eq2.5}) can be proved by a modification of the argument in the proof of Theorem 3.4 in \cite{sa14}. For completeness, we present a detailed proof as follows.

Fixed any three functions   $f^1,f^2,f^3$ in
$$\mathbb D_\infty:=\big\{ f\in {\mathscr D}(\Delta_{\mathscr E})\cap \Lip(X)\cap L^\infty(X)|\ \Delta f\in W^{1,2}(X)\big\},$$
we choose the polynomial $\Phi:\mathbb R^3\to \mathbb R$ defined by (the same in \cite{sa14})
$$\Phi(\mathbf{ f}):=\lambda f^1+(f^2-a)(f^3-b)-ab,\quad \lambda,a,b\in\mathbb R,$$
where $\mathbf{ f}=(f^1,f^2,f^3).$
Now we have $\Phi(\mathbf{f})\in\mathbb D_\infty$ by  \cite[Lemma 3.2]{sa14}, and
$$\gamma_2\big(\Phi(\mathbf{f})\big)\gs K|\nabla \Phi(\mathbf{f})|^2+\nu\cdot\big(\Delta\Phi(\mathbf{f})\big)^2,\quad \mu-a.e.\ {\rm in}\ X,$$
where $\nu:=1/N$, and $\gamma_2(f):=\frac{1}{2}\Delta^R(|\nabla f|^2)-\ip{\nabla f}{\nabla \Delta f}$.
According to \cite{sa14}, we know that for $\mu$-almost every $x\in X$ the inequality holds for every $(\lambda,a,b)\in\mathbb R^3$. Up to a $\mu$-negligible set, for every $x\in X$, we can take $a:=f^2(x)$ and $b:=f^3(x)$.
Thus, we obtain
\begin{equation}\label{eq++}
\begin{split}
\lambda^2\cdot&\gamma_2(f^1)+4\lambda\cdot H[f^1](f^2,f^3)+2\Big(|\nabla f^2|^2\cdot|\nabla f^3|^2+\ip{\nabla f^2}{\nabla f^3}^2 \Big)\\
&\gs K\lambda^2\cdot |\nabla f^1|^2+ \nu\cdot\Big(\lambda\cdot \Delta f^1+2\ip{\nabla f^2}{\nabla f^3} \Big)^2\\
&= K\lambda^2\cdot |\nabla f^1|^2+ \nu\cdot\Big(\lambda^2\cdot(\Delta f^1)^2+4\cdot\lambda\cdot\Delta f^1\ip{\nabla f^2}{\nabla f^3}+4 \cdot\ip{\nabla f^2}{\nabla f^3}^2\Big),
\end{split}
\end{equation}
where $ H[f](g,h):=\frac{1}{2}\cdot [\ip{\nabla g}{\nabla\ip{\nabla f}{\nabla h}}+ \ip{\nabla h}{\nabla\ip{\nabla f}{\nabla g}}  -\ip{\nabla f}{\nabla\ip{\nabla g}{\nabla h}}]$. We can rewrite inequality (\ref{eq++}) as in the following form
 \begin{equation*}
\begin{split}
0&\ls \lambda^2\cdot\Big(\gamma_2(f^1)-K\cdot |\nabla f^1|^2-\nu\cdot(\Delta f^1)^2\Big)+\lambda\cdot\Big(4 H[f^1](f^2,f^3)-4\nu\cdot\Delta f^1\ip{\nabla f^2}{\nabla f^3}\Big)\\
&\quad+2\Big(|\nabla f^2|^2\cdot|\nabla f^3|^2+\ip{\nabla f^2}{\nabla f^3}^2 \Big)-4\nu \cdot\ip{\nabla f^2}{\nabla f^3}^2\\
&:=\mathcal A\cdot \lambda^2+\mathcal B\cdot\lambda +\mathcal C.
\end{split}
\end{equation*}
Since $\lambda$ is arbitrary, the coefficients $\mathcal A,\mathcal B$ and $\mathcal C$ satisfy $\mathcal B^2\ls 4\mathcal A\mathcal C$. I.e.,
\begin{equation}\label{eq2.7}
\begin{split}
16&\cdot\Big(H[f^1](f^2,f^3)-\nu\cdot \Delta f^1\ip{\nabla f^2}{\nabla f^3}\Big)^2\\
&\ls 4\Big(\gamma_2(f^1)-K|\nabla f^1|^2-\nu(\Delta f^1)^2\Big)\cdot  \bigg[2\Big(|\nabla f^2|^2\cdot|\nabla f^3|^2+\ip{\nabla f^2}{\nabla f^3}^2 \Big)-4\nu \!\cdot\!\ip{\nabla f^2}{\nabla f^3}^2 \bigg].
\end{split}
\end{equation}
By using
$$ |\nabla f^2|^2\cdot|\nabla f^3|^2+\ip{\nabla f^2}{\nabla f^3}^2\ls 2\cdot|\nabla f^2|^2\cdot|\nabla f^3|^2 $$
and noting that $\mathcal A\gs0$ (from (\ref{eq2.4})), the inequality (\ref{eq2.7}) implies that
\begin{equation}\label{eq2.8}
\begin{split}
\Big(H[f^1]&(f^2,f^3)-\nu\cdot \Delta f^1\ip{\nabla f^2}{\nabla f^3}\Big)^2\\
&\ls \Big(\gamma_2(f^1)-K|\nabla f^1|^2-\nu(\Delta f^1)^2\Big)\cdot  \bigg[|\nabla f^2|^2\cdot|\nabla f^3|^2-\nu \!\cdot\!\ip{\nabla f^2}{\nabla f^3}^2 \bigg].
\end{split}
\end{equation}
Now we take $f^1=f^2=f^3=f$ and conclude that
$$\Big(\frac{1}{2}\ip{\nabla f}{\nabla |\nabla f|^2}-\nu\cdot\Delta f\cdot|\nabla f|^2\Big)^2\ls (1-\nu)\cdot |\nabla f|^4\cdot\Big(\gamma_2(f)-K|\nabla f|^2-\nu(\Delta f)^2\Big).$$
This implies (\ref{eq2.5})  where $|\nabla f|\not=0.$ The proof is complete.
\end{proof}
\begin{rem}\label{rem1}
 If $N=\infty$, we have $1/N=0$ and $N/(N-1)=1$, and that the inequality $(\ref{eq2.5})$ is
\begin{equation}\label{eq2.9}
\frac{1}{2} \Delta^{R}(|\nabla f|^2)\gs  \la\nabla\Delta f,\nabla f \ra+K|\nabla f|^2+\Big(\frac{\la \nabla f,\nabla |\nabla f|^2\ra}{2|\nabla f|^2}\Big)^2.
\end{equation}
It is proved in \cite{sa14} that
\begin{equation*}
\frac{1}{2} \Delta^{R}(|\nabla f|^2)\gs  \la\nabla\Delta f,\nabla f \ra+K|\nabla f|^2+\frac{|\nabla |\nabla f|^2|^2}{4\cdot|\nabla f|^2}.
\end{equation*}
 This is stronger than $(\ref{eq2.9})$, since  $|\la\nabla f,\nabla |\nabla f|^2\ra|\ls |\nabla f|\cdot|\nabla |\nabla f|^2|$.
\end{rem}

\section{Hamilton's gradient estimates}
\hskip\parindent
We will prove the main Theorem \ref{main} in this section. Let us begin from the following Lemma.

\begin{lem}\label{lem3.1}
Let $(X,d,\mu)$ be  a   $RCD^*(K, \infty)$ space with some $K\ls0$.
Assume that
$u(x,t): X\times [0,T]\to\rr$, $T\ls \infty$, is  a solution of heat equation $\partial_tu=\Delta u$ with initial value $u_0(x)\in L^2(X)\cap L^\infty(X)$ such that $u_0\gs0$.
For any $\epsilon>0$, we have $\frac{|\nabla u|^2}{u+\epsilon}\in \mathscr D(\Delta^*)$ for all $t\in (0,T]$
and
\begin{equation}\label{eq3.1}
\Delta^*\Big(\frac{|\nabla u|^2}{u+\epsilon}\Big)- \frac{\partial}{\partial t}\Big(\frac{|\nabla u|^2}{u+\epsilon}\Big)\cdot\mu\gs
2K\frac{|\nabla u|^2}{u+\epsilon}\cdot\mu.
\end{equation}
\end{lem}
\begin{proof}Fix any $0<t<T$. By using Lemma \ref{lem2.2}, we get that $u(\cdot,t)\in {\rm Lip}(X)$ for all $t>0$.
Since $u_0\in L^2(X)\cap L^\infty(X)$, we get $u\in L^\infty(X)\cap \mathscr D(\Delta_{\mathscr E})\cap \Lip(X)$ and $\Delta u\in W^{1,2}(X)$ for all $t>0$.

Fix any $\epsilon>0$ and any $\ell>0$, we first claim that $(u+\epsilon)^{-\ell}\in \mathscr D(\Delta)\cap {\rm Lip}(X)$. Since $u\in  \Lip(X)$ and $u+\epsilon\gs \epsilon$, we have $(u+\epsilon)^{-\ell}\in{\rm Lip}(X)$.
By using Lemma \ref{lem2.1}(i), we have $(u+\epsilon)^{-\ell}\in \mathscr D(\Delta^*)$ and
\begin{equation}\label{eq3.2}
\Delta^*(u+\epsilon)^{-\ell}=-\ell(u+\epsilon)^{-\ell-1}\Delta^*u+\ell(\ell+1)\cdot(u+\epsilon)^{-\ell-2}|\nabla u|^2\cdot \mu.
\end{equation}
Since $\Delta^* u=\Delta u\cdot \mu$ with $\Delta u\in L^2(X)$, $|\nabla u|\in L^\infty(X)$
and $(u+\epsilon)^{-\ell}\in L^\infty(X)$, we have $(u+\epsilon)^{-\ell}\in \mathscr D(\Delta)$.

By using $u\in L^\infty(X)\cap \mathscr D(\Delta_{\mathscr E})\cap \Lip(X)$ and
$\Delta u\in W^{1,2}(X)$ and Lemma \ref{lem2.3}, we have $|\nabla u|^2\in \mathscr D(\Delta^*)\cap W^{1,2}(X).$
According to Lemma \ref{lem2.1}(ii) (with condition (c)), we conclude that $\frac{|\nabla u|^2}{(u+\epsilon)^\ell}\in \mathscr D(\Delta^*)$ for all $t>0$ and
$$\Delta^*\Big(\frac{|\nabla u|^2}{(u+\epsilon)^\ell}\Big)= \frac{\Delta^*(|\nabla u|^2)}{(u+\epsilon)^\ell}+|\nabla u|^2\cdot \Delta^* [(u+\epsilon)^{-\ell}]-2\ell\frac{\la\nabla u,\nabla |\nabla u|^2\ra}{(u+\epsilon)^{\ell+1}}\cdot\mu.$$
By using Lemma \ref{lem2.3} and (\ref{eq3.2}), we have
\begin{equation}\label{eq3.3}
\Delta^S\Big(\frac{|\nabla u|^2}{(u+\epsilon)^\ell}\Big)= \frac{\Delta^S(|\nabla u|^2)}{(u+\epsilon)^\ell}\gs0
\end{equation}
and
\begin{equation}\label{eq3.4}
\Delta^R\Big(\frac{|\nabla u|^2}{(u+\epsilon)^\ell}\Big)= \frac{\Delta^R(|\nabla u|^2)}{(u+\epsilon)^\ell}-\ell\frac{|\nabla u|^2\Delta u}{(u+\epsilon)^{\ell+1}}+ \ell(\ell+1)\cdot\frac{|\nabla u|^4 }{(u+\epsilon)^{\ell+2}}-2\ell\frac{\la\nabla u,\nabla |\nabla u|^2\ra}{(u+\epsilon)^{\ell+1}}.
\end{equation}
Note that
$$\frac{\partial}{\partial t}\Big(\frac{|\nabla u|^2}{(u+\epsilon)^\ell}\Big)=\frac{\partial_t|\nabla u|^2}{(u+\epsilon)^\ell}-\ell\frac{|\nabla u|^2\partial_t u}{(u+\epsilon)^{\ell+1}}.$$
Using  $\Delta u=\partial_t u $, $\partial_t|\nabla u|^2:=\frac{\partial}{\partial t}|\nabla u|^2=2\ip{\nabla u}{\nabla \Delta u}$ and (\ref{eq3.4}), we have, for $\mu$-a.e. $\ x\in X$,
\begin{equation}\label{eq3.5}
\Delta^R\Big(\frac{|\nabla u|^2}{(u+\epsilon)^\ell}\Big)- \frac{\partial}{\partial t}\Big(\frac{|\nabla u|^2}{(u+\epsilon)^\ell}\Big)= \frac{\Delta^R(|\nabla u|^2)-2\ip{\nabla u}{\nabla \Delta u}}{(u+\epsilon)^\ell}+ \ell(\ell+1)\frac{|\nabla u|^4 }{(u+\epsilon)^{\ell+2}}-2\ell\frac{\la\nabla u,\nabla |\nabla u|^2\ra}{(u+\epsilon)^{\ell+1}}.
\end{equation}

Next, we take $\ell=1$ in this proof.

At the points where $|\nabla u|=0$, using (\ref{eq2.4}) and (\ref{eq3.5}), we have
$$\Delta^R\Big(\frac{|\nabla u|^2}{u+\epsilon}\Big)- \frac{\partial}{\partial t}\Big(\frac{|\nabla u|^2}{u+\epsilon}\Big)\gs 0,$$
where we have used $|\la\nabla u,\nabla |\nabla u|^2\ra|\ls |\nabla u|\cdot|\nabla|\nabla u|^2 |=0$.

At the points where $|\nabla u|\not=0$, by using  (\ref{eq3.5}) for $\ell=1$, we have
\begin{equation}\label{eq3.6}
\begin{split}
\Delta^R\Big(\frac{|\nabla u|^2}{u+\epsilon}\Big)- \frac{\partial}{\partial t}\Big(\frac{|\nabla u|^2}{u+\epsilon}\Big)&\ =\ \frac{2}{u+\epsilon}\bigg(
\frac{\Delta^R(|\nabla u|^2)-2\ip{\nabla u}{\nabla \Delta u}}{2}+ \frac{|\nabla u|^4 }{(u+\epsilon)^{2}}-\frac{\la\nabla u,\nabla |\nabla u|^2\ra}{u+\epsilon}\bigg).
\end{split}
\end{equation}
By using (\ref{eq2.5}) with $N=\infty$ to the function $u$ (see also Remark \ref{rem1}), we get
$$ \frac{1}{2}\Delta^{R}(|\nabla u|^2)\gs  \la\nabla\Delta u,\nabla u \ra+K|\nabla u|^2+\Big(\frac{\la \nabla u,\nabla |\nabla u|^2\ra}{2|\nabla u|^2}\Big)^2.$$
By combining this and the equation (\ref{eq3.6}), we have
\begin{equation*}
\Delta^R\Big(\frac{|\nabla u|^2}{u+\epsilon}\Big)- \frac{\partial}{\partial t}\Big(\frac{|\nabla u|^2}{u+\epsilon}\Big)\gs\frac{2}{u+\epsilon}  \Big(K|\nabla u|^2+\mathscr A\Big)
\end{equation*}
  where
 $$\mathscr A:=\bigg(\frac{\la\nabla u,\nabla |\nabla u|^2\ra}{2|\nabla u|^2}\bigg)^2+\frac{|\nabla u|^4}{(u+\epsilon)^2}-\frac{\la\nabla u,\nabla |\nabla u|^2\ra}{u+\epsilon}=\bigg(\frac{\la\nabla u,\nabla |\nabla u|^2\ra}{2|\nabla u|^2}-\frac{|\nabla u|^2}{u+\epsilon}\bigg)^2\gs0.$$
 Hence, we have, for $\mu$-a.e. $\ x\in X$,
\begin{equation}\label{eq3.7}
\Delta^R\Big(\frac{|\nabla u|^2}{u+\epsilon}\Big)- \frac{\partial}{\partial t}\Big(\frac{|\nabla u|^2}{u+\epsilon}\Big)\gs  \frac{2K|\nabla u|^2}{u+\epsilon}.
\end{equation}
The combination of (\ref{eq3.3}) and (\ref{eq3.7}) implies the desired (\ref{eq3.1}). The proof is complete.
\end{proof}
\noindent \textbf{Remark.}  In the above proof, it is crucial that the singular part is nonnegative.

\begin{lem}\label{lem3.2}
Let $(X,d,\mu)$ be  a   $RCD^*(K,\infty)$ space with some $K\ls0$.
Assume that
$u(x,t): X\times [0,T]\to\rr$, $T\ls \infty$, is  a solution of heat equation
$\partial_tu=\Delta u$ with initial value $u_0(x)\in L^2(X)\cap L^\infty(X)$ such that $u_0\gs0$.
For any $\epsilon>0$, we set
$$P_\epsilon(x,t):=\varphi(t)\frac{|\nabla u|^2}{u+\epsilon}+(u+\epsilon)\log (u+\epsilon),$$
where $\varphi(t)=\frac{t}{1-2Kt}.$ Then we have $P_\epsilon(\cdot,t)\in \mathscr D(\Delta^*)$ for all $t\in (0,T]$ and
$$\frac{\partial}{\partial t}P_\epsilon(x,t)\cdot\mu\ls \Delta^* P_\epsilon(x,t).$$
\end{lem}
\begin{proof}
From Lemma \ref{lem3.1}, we have $\frac{|\nabla u|^2}{u+\epsilon}\in \mathscr D(\Delta^*)$ for all $t>0$
and
\begin{equation}\label{eq3.8}
\Delta^*\Big(\frac{|\nabla u|^2}{u+\epsilon}\Big)- \frac{\partial}{\partial t}\Big(\frac{|\nabla u|^2}{u+\epsilon}\Big)\cdot\mu\gs
2K\frac{|\nabla u|^2}{u+\epsilon}\cdot\mu.
\end{equation}
Since $u\in  \Lip(X)$ and $u+\epsilon\gs \epsilon$, we have $(u+\epsilon)\log(u+\epsilon)\in{\rm Lip}(X)$.
By using Lemma \ref{lem2.1}(i), we have $(u+\epsilon)\log(u+\epsilon)\in \mathscr D(\Delta^*)$ and
\begin{equation}\label{eq3.9}
\begin{split}
\Delta^*\big((u+\epsilon)\log(u+\epsilon)\big)&=\big(1+\log(u+\epsilon)\big)\Delta^*u+ \frac{|\nabla u|^2}{u+\epsilon}\cdot \mu\\
&=\Big[\big(1+\log(u+\epsilon)\big)\Delta u+ \frac{|\nabla u|^2}{u+\epsilon}\Big]\cdot \mu,
\end{split}
\end{equation}
where we have used $\Delta^* u=\Delta u\cdot \mu$ with $\Delta u\in L^2(X)$. Hence we have
\begin{equation}\label{eq3.10}
\Delta^*\big((u+\epsilon)\log(u+\epsilon)\big)- \frac{\partial}{\partial t}(u+\epsilon)\log(u+\epsilon)\cdot\mu= \frac{|\nabla u|^2}{u+\epsilon}\cdot \mu,
\end{equation}
since $\partial_tu=\Delta u.$ By applying (\ref{eq3.9}) and (\ref{eq3.10}), we have
$$\Delta^* P_\epsilon(x,t)-\frac{\partial}{\partial t}P_\epsilon(x,t)\cdot\mu\gs\Big(1-\varphi'+2K\varphi\Big)\cdot \frac{|\nabla u|^2}{u+\epsilon}\cdot\mu.$$
Together with $ 1-\varphi'(t)+2K\varphi(t)=-2Kt/(1-2Kt)^2\gs0$ (since $K\ls0$), we get $\Delta^* P_\epsilon(x,t)-\frac{\partial}{\partial t}P_\epsilon(x,t)\cdot\mu\gs0$.
The proof is finished.
\end{proof}



We also need the the following lemma, which can be found \cite{g87} or \cite[Lemma 2 on Page 165]{sy94}.
\begin{lem}\label{lem3.3}
Let a function $v(x,t)$ on $X\times[0,\infty)$ be a weak sub-solution of the heat equation and $v(x,0)=v_0(x)$
, i.e., $v(x,t)$ satisfies $\partial_tv\ls \Delta v$ in the sense of distribution. Namely, the following inequality
$$-\int_0^\infty\int_X\ip{\nabla v}{\nabla \varphi}d\mu dt\gs \int_0^\infty\int_X \partial_tv\cdot\varphi d\mu dt$$
 holds for any nonnegative function $\varphi(x,t)\in {\rm Lip}_0(X\times[0,\infty)).$

Let a function $g(x,t)$ be of class $C^1$ in $t$ and Lipschitz continuous in $x$, and satisfy
$$g\ls 0,\qquad \frac{\partial}{\partial t}g+\frac{1}{4}|\nabla g|^2=0.$$
For any $T>0$, if we assume
$$\int_0^T\int_Xe^{\frac{g(x,t)}{4}}\cdot v^2(x,t)d\mu dt<\infty,$$
then, for any $R>0$ and $x_0\in X$, we have
\begin{equation}\label{eq4.6}
\int_{B(x_0,R)}e^{\frac{g(x,T)}{4}}v^2(x,T)d\mu  \ls   \int_{X}e^{\frac{g(x,0)}{4}}v_0^2(x)d\mu.
\end{equation}
  \end{lem}
\begin{proof}
See the proof in \cite[Page 166-167]{sy94} and take $\epsilon=4$ there.
\end{proof}

Now we are in the position to prove the main theorem.
\begin{proof}[Proof of Theorem \ref{main}]
By replacing $u$ by $u/2M$, we can assume that $0<u\ls1/2$.

\noindent\textbf{(i).} First, we consider the case where the initial data $u_0\in L^2(X)$.
For an arbitrarily fixed $\epsilon\in (0,1/2)$, set $u_\epsilon:=u+\epsilon$ and
$$ P_\epsilon(x,t):=\varphi(t)\frac{|\nabla u|^2}{u_\epsilon}+u_\epsilon\log\big(u_\epsilon\big), $$
where $\varphi(t):=t/(1-2Kt).$
By Lemma \ref{lem3.2}, we have
$$\frac{\partial}{\partial t}P_\epsilon\cdot\mu\ls \Delta^* P_\epsilon,$$
and hence $P^+_\epsilon$ is a weak sub-solution of the heat equation,
  where $P_\epsilon^+(x,t):=\max\{P_\epsilon(x,t),0\}.$

Fix a $T>0$. The volume growth condition \eqref{vol-growth} implies that there exists a small constant $\delta=\delta(c_1,c_2)>0$ such that
$$\int_Xe^{-\frac{d^2(x_0,x)}{\delta}}d\mu<\infty$$
for some $x_0\in X$.
We denote
$$g(x,t):=-\frac{d^2(x_0,x)}{T-t+\delta/4}\ls-\frac{d^2(x_0,x)}{\delta/4}.$$
Then
$$ \frac{\partial}{\partial t}g+\frac{1}{4}|\nabla g|^2=0.$$
 From Lemma \ref{lem2.2}, we have
$$P^+_\epsilon\ls \varphi(t)\frac{|\nabla u|^2}{u_\epsilon}\ls \varphi(t)\cdot\frac{C'}{t}\cdot \epsilon^{-1}
\ls  C' \epsilon^{-1}$$ for some constant $C'=C'(K,T)>0$, {where we
have used $K\ls0$ and $u,u_\epsilon\ls1$.} Hence,
\begin{equation*}
\int_0^T\!\int_Xe^{\frac{g(x,t)}{4}}(P^+_\epsilon)^2(x,t) d\mu\ls T\cdot
(C' \epsilon^{-1})^2\cdot\int_Xe^{-\frac{d^2(x_0,x)}{\delta}}d\mu<\infty.
\end{equation*}
 Now the functions $P^+_\epsilon$ and $g$ meet all of assumptions
of Lemma \ref{lem3.3}. Hence  we  conclude that
\begin{equation}\label{eq3.12}
\int_{B(x_0,R)}\big(P^+_\epsilon(x,T)\big)^2\cdot e^{g/4}d\mu\ls \int_{X}\big(P^+_\epsilon(x,0)\big)^2\cdot e^{g/4}d\mu.
\end{equation}
Notice that $P^+_\epsilon(x,0)=0$, we have obtained, from equation \eqref{eq3.12},
 $$\int_{B(x_0,R)}\big(P^+_\epsilon(x,T)\big)^2\cdot e^{g/4}d\mu=0, \qquad{\rm and\ \ then},\quad  P^+_\epsilon(x,T)=0.$$
Therefore, we have $P_\epsilon(x,T)\ls0.$ That is,
\begin{equation}\label{eq3.13}
\varphi(T)\frac{|\nabla u|^2}{u_\epsilon^2}\ls -\log u_\epsilon.
\end{equation}
Letting $\epsilon\to0^+$, we have
\begin{equation*}
\varphi(T)\frac{|\nabla u|^2}{u^2}\ls \log \Big(\frac{1}{u}\Big).
\end{equation*}
This proves the desired Hamilton estimate for the case where the initial data $u_0\in L^2(X)$.\\

\noindent\textbf{(ii).}\ \ The proof of the case where the initial data $u_0\in L^q(X)$ for some $q\in[1,\infty)$.

It is sufficient to consider the case where $q\gs2$. Indeed, if $q<2$,  the facts $u_0(x) \in L^q(X)$ and $0\ls u(x,t)\ls M$ implies that $u_0\in L^2(X).$
 Take a sequence of functions $\tilde{u}_j\in L^2(X)\cap L^\infty(X)$ such that
 $$\tilde {u}_j\overset{ L^q}{\to} u_0,\qquad \tilde {u}_j\overset{w-L^\infty}{\to} u_0 \qquad {\rm and} \qquad  0\ls \tilde{u}_j  \ls M.$$

 Fix an $\epsilon>0$ and let $u_j(x,t):=H_t\tilde{u}_j(x)$ and $u_{j,\epsilon}:=u_j+\epsilon.$
 Now we have
 $$\epsilon\ls u_{j,\epsilon}(x,t)\ls M+\epsilon, \ j=1,2,\cdots.$$

For any fixed $T>0$, $R>0$, and any Lipschitz function $0\ls f(x)\in
\Lip_0(B_R(x_0))$,  from \eqref{eq3.13}, we get
 \begin{equation}\label{eq4.9}
\int_{ B_{R}(x_0)}\varphi(T)|\nabla u_j|^2\cdot fd\mu   \ls -\int_{ B_{R}(x_0)}u_{j,\epsilon}^2\cdot\log u_{j,\epsilon}\cdot fd\mu, \ \ j=1,2,\cdots.
\end{equation}
The fact $\tilde {u}_j\overset{ L^q}{\to} u_0$  implies that
$u_j(x,T)\overset{ L^q}{\to} u(x,T)$ by $L^q$-boundedness of $H_t$.
Thus, $u_j(x,T)\to u(x,T)$ for almost all $x \in X$, up to a
subsequence. Notice that $u^2_{j,\epsilon}\cdot\log u_{j,\epsilon}$
is uniformly  bounded on $B_R(x_0)$, by dominated convergence
theorem, we have
\begin{equation}\label{eq4.10}
\lim_{j\to\infty}\int_{ B_{R}(x_0)}u_{j,\epsilon}^2\cdot\log u_{j,\epsilon}\cdot fd\mu=\int_{ B_{R}(x_0)}u_{\epsilon}^2\cdot\log u_{\epsilon}\cdot fd\mu.
\end{equation}
Notice that $q\gs2$ and $u_j(x,T)\overset{ L^q}{\to} u(x,T)$, it is clear that $u_j(x,T)\overset{ L^2(B_R(x_0))}{\to} u(x,T)$. Therefore, by the lower semi-continuity of energy and \eqref{eq4.9}-\eqref{eq4.10}, we have
  \begin{equation*}
\int_{ B_{R}(x_0)}\varphi(T)|\nabla u|^2\cdot fd\mu\ls \liminf_{j\to\infty}\int_{ B_{R}(x_0)}\varphi(T)|\nabla u_j|^2\cdot fd\mu   \ls -\int_{ B_{R}(x_0)}u_{\epsilon}^2\cdot\log u_{\epsilon}\cdot fd\mu.
\end{equation*}
By the arbitrariness of $f$ and $R$, we get
$ \varphi(T) |\nabla u|^2\ls -u_\epsilon^2\cdot\log u_\epsilon.$ Letting $\epsilon\to0^+$, we have
\begin{equation*}
\varphi(T) |\nabla u|^2 \ls u^2\cdot\log \Big(\frac{1}{u}\Big).
\end{equation*}
This is the desired Hamilton estimate.
\end{proof}

\section{Monotonicity for heat equations}
\hskip\parindent
We will prove the main Theorem \ref{thm1.3} in this section.
\begin{lem}\label{lem4.1}
Let $(X,d,\mu)$ be  a  proper $RCD^*(0, N)$ space with some $N\in[1,\infty]$.
Assume that
$u(x,t): X\times [0,T]\to\rr$, $T\ls \infty$, is  a solution of heat equation
$\partial_tu=\Delta u$ with initial value $0\le u_0(x)\in L^2(X)\cap L^\infty(X)$.
For any $\epsilon>0$, we have $w_\epsilon(x,t)\in \mathscr D(\Delta^*)$ for all $t\in (0,T]$
and
\begin{equation}\label{eq4.1}
\Delta^*w_\epsilon- \frac{\partial}{\partial t}w_\epsilon\cdot\mu\gs\bigg[-2\frac{\ip{\nabla w_\epsilon}{\nabla u}}{u_\epsilon}+\frac{2}{Nu^2_\epsilon}\Big(\Delta u-\frac{|\nabla u|^2}{u_\epsilon}\Big)^2\bigg]\cdot\mu,
\end{equation}
where $u_\epsilon:=u+\epsilon$ and
$$w_\epsilon(x,t):=\frac{|\nabla u|^2}{u^2_\epsilon}-2\frac{\partial_tu}{u_\epsilon}.$$
Here and in the sequel, if $N=\infty$, then $\frac{1}{N}=0$ and $\frac{N-1}{N}=1.$
\end{lem}
\begin{proof}
By \eqref{eq3.3} and \eqref{eq3.5} the proof in the Lemma 3.1 and taking $\ell=2$ there,
we have $|\nabla u|^2/u^2_\epsilon\in \mathscr D(\Delta^*)$ with
 \begin{equation}\label{eq4.2}
 \Delta^S(|\nabla u|^2/u^2_\epsilon)\gs0
 \end{equation}
 and,
  for $\mu$-a.e. $\ x\in X$,
\begin{equation}\label{eq4.3}
\Delta^R\Big(\frac{|\nabla u|^2}{u_\epsilon^2}\Big)- \frac{\partial}{\partial t}\Big(\frac{|\nabla u|^2}{u_\epsilon^2}\Big)= \frac{\Delta^R(|\nabla u|^2)-2\ip{\nabla u}{\nabla \Delta u}}{u_\epsilon^2}+ 6\frac{|\nabla u|^4 }{u_\epsilon^{4}}-4\frac{\la\nabla u,\nabla |\nabla u|^2\ra}{u_\epsilon^{3}}.
\end{equation}
Since $\partial_tu$ also solves the heat equation, we have $\partial_tu\in  \Lip(X)$.  From Lemma 2.1(ii)
(with condition (a)), we have $\frac{\partial_t u}{u_\epsilon}\in \mathscr D(\Delta^*)$ and
$$\Delta^*\Big(\frac{\partial_t u}{u_\epsilon}\Big)= \Delta^*(\partial_t u)\cdot u^{-1}_\epsilon+\partial_tu\big(-u^{-2}_\epsilon\cdot \Delta^*u+2u^{-3}_\epsilon|\nabla u|^2\cdot\mu\big)-2\la\nabla \partial_tu,u^{-2}_\epsilon\nabla u\ra\cdot\mu.$$
By noting $\Delta^* u=\partial_tu\cdot\mu$ and $\Delta^*(\partial_tu)=\Delta (\partial_tu)\cdot\mu$, we have
\begin{equation}\label{eq4.4}
\Delta^*\Big(\frac{\partial_t u}{u_\epsilon}\Big)-\frac{\partial}{\partial t}\Big(\frac{\partial_t u}{u_\epsilon}\Big)\cdot\mu=2\partial_t u\cdot u^{-3}_\epsilon|\nabla u|^2\cdot\mu-2 u^{-2}_\epsilon\la\nabla \partial_t u,\nabla u\ra\cdot\mu.
\end{equation}
Now $w_\epsilon\in \mathscr D(\Delta^*)\subset W^{1,2}_{\rm loc}(X)$, we have
$$\ip{\nabla w_\epsilon}{\nabla u}=u^{-2}_\epsilon\ip{\nabla |\nabla u|^2}{\nabla u}-2 u^{-3}_\epsilon|\nabla u|^4-2u^{-1}_\epsilon\ip{\nabla \partial_tu}{\nabla u}+2 u^{-2}_\epsilon \partial_tu\cdot|\nabla u|^2.$$
Therefore, by combining with (\ref{eq4.2}), (\ref{eq4.3}) and (\ref{eq4.4}), we obtain that $\Delta^Sw_\epsilon\gs0$ and, for $\mu$-a.e. $\ x\in X$,
\begin{equation}\label{eq4.5}
\begin{split}
\Delta^Rw_\epsilon-& \frac{\partial}{\partial t}w_\epsilon+2\frac{\ip{\nabla w_\epsilon}{\nabla u}}{u_\epsilon}\\
&= \frac{\Delta^R(|\nabla u|^2)-2\ip{\nabla u}{\nabla \Delta u}}{u_\epsilon^2}+ 6\frac{|\nabla u|^4 }{u_\epsilon^{4}}-4\frac{\la\nabla u,\nabla |\nabla u|^2\ra}{u_\epsilon^{3}}\\
&\quad-2\big[ 2\partial_t u\cdot u^{-3}_\epsilon|\nabla u|^2-2 u^{-2}_\epsilon\la\nabla \partial_t u,\nabla u\ra\big]\\
&\quad+2\Big[u^{-3}_\epsilon\ip{\nabla |\nabla u|^2}{\nabla u}-2 u^{-4}_\epsilon|\nabla u|^4-2u^{-2}_\epsilon\ip{\nabla \partial_tu}{\nabla u}+2 u^{-3}_\epsilon \partial_tu\cdot|\nabla u|^2\Big]\\
&=\frac{\Delta^R(|\nabla u|^2)-2\ip{\nabla u}{\nabla \Delta u}}{u_\epsilon^2}+ 2\frac{|\nabla u|^4 }{u_\epsilon^{4}}-2\frac{\la\nabla u,\nabla |\nabla u|^2\ra}{u_\epsilon^{3}}.
\end{split}
\end{equation}
At the points where $|\nabla u|=0$, using (\ref{eq2.4}) and (\ref{eq4.5}), we have
$$\Delta^Rw_\epsilon-  \frac{\partial}{\partial t}w_\epsilon+2\frac{\ip{\nabla w_\epsilon}{\nabla u}}{u_\epsilon}\gs \frac{2}{Nu^2_\epsilon}(\Delta u)^2,$$
where we have used $|\la\nabla u,\nabla |\nabla u|^2\ra|\ls |\nabla u|\cdot|\nabla|\nabla u|^2 |=0$.

At the points where $|\nabla u|\not=0$, using (\ref{eq2.5}) and (\ref{eq4.5}), we have
$$\Delta^Rw_\epsilon-  \frac{\partial}{\partial t}w_\epsilon+2\frac{\ip{\nabla w_\epsilon}{\nabla u}}{u_\epsilon}\gs \frac{2}{u^2_\epsilon}\cdot \mathscr A,$$
where
$$\mathscr A:=\frac{(\Delta u)^2}{N}+\frac{N}{N-1}\Big(\frac{\la \nabla u,\nabla |\nabla u|^2\ra}{2|\nabla u|^2}-\frac{\Delta u}{N}\Big)^2+\frac{|\nabla u|^4}{u^2_\epsilon}-\frac{\la \nabla u,\nabla |\nabla u|^2\ra}{u_\epsilon}.$$
Let us set $  B_1:=\frac{\la \nabla u,\nabla |\nabla u|^2\ra}{|\nabla u|^2}$ and $B_2:=\frac{|\nabla u|^2}{u_\epsilon}$. We get
\begin{equation}
\begin{split}
\mathscr A&=\frac{(\Delta u)^2}{N}+\frac{N}{N-1}\Big(\frac{B_1^2}{4}-\frac{B_1\cdot\Delta u}{N}+\frac{(\Delta u)^2}{N^2}\Big)+ B_2^2-B_1B_2\\
&=\frac{(\Delta u)^2}{N-1}+\frac{N}{N-1}\bigg[\frac{B_1^2}{4}-B_1\cdot\Big(\frac{\Delta u}{N}+\frac{N-1}{N}B_2\Big) \bigg]+ B_2^2\\
&\gs \frac{(\Delta u)^2}{N-1}+ B_2^2+\frac{N}{N-1}\bigg[-\Big(\frac{\Delta u}{N}+\frac{N-1}{N}B_2\Big)^2 \bigg]\\
&= \frac{(\Delta u)^2}{N-1}+ B_2^2-\Big(\frac{(\Delta u)^2}{N(N-1)}+\frac{2 B_2\cdot\Delta u}{N}+\frac{N-1}{N}B_2^2\Big)\\
&=\frac{(\Delta u-B_2)^2}{N}.
\end{split}
\end{equation}
Therefore, we have
$$\Delta^Rw_\epsilon-  \frac{\partial}{\partial t}w_\epsilon+2\frac{\ip{\nabla w_\epsilon}{\nabla u}}{u_\epsilon}\gs \frac{2}{Nu^2_\epsilon}\cdot (\Delta u-B_2)^2.$$
Together with $\Delta^Sw_\epsilon\gs0$, we have proved the lemma.
\end{proof}
 \begin{lem}\label{lem4.2}
Let $(X,d,\mu)$ be  a   $RCD^*(0, N)$ space with some $N\in[1,\infty)$.
Assume that
$u(x,t): X\times [0,T]\to\rr$, $T\ls \infty$, is  a solution of heat equation
$\partial_tu=\Delta u$ with initial value $u_0(x)\in L^2(X)\cap L^\infty(X)$ such that $u_0\gs0$.
For each $\epsilon>0$, we set
$$W_\epsilon(x,t):=\tau\cdot w_\epsilon-\log u_\epsilon-\frac{N}{2}\log(4\pi\tau)-N,$$
where $\tau=\tau(t)>0$ with $\frac{d\tau}{dt}=1$ and $u_\epsilon,w_\epsilon$ are given in Lemma \ref{lem4.1}. Then
$W_\epsilon\in\mathscr D(\Delta^*)$ and
\begin{equation}
\Delta^*W_\epsilon- \frac{\partial}{\partial t}W_\epsilon\cdot\mu\gs-2\frac{\ip{\nabla W_\epsilon}{\nabla u}}{u_\epsilon}\cdot\mu.
\end{equation}
\end{lem}
\begin{proof}
From Lemma 4.1 and that $\log u_\epsilon\in \mathscr D(\Delta^*)$, we deduce that
$W_\epsilon\in\mathscr D(\Delta^*)$. Moreover, $\Delta^SW_\epsilon=\tau\cdot\Delta^Sw_\epsilon\gs0$.
By directly calculating and using Lemma \ref{lem4.1}, we get
\begin{equation*}
\begin{split}
\Big(\Delta^R-\frac{\partial}{\partial t}\Big)W_\epsilon&=\tau\cdot\Big(\Delta^R-\frac{\partial}{\partial t}\Big)w_\epsilon-w_\epsilon+\frac{|\nabla u|^2}{u^2_\epsilon}+\frac{N}{2\tau}\\
&\gs -2\frac{\ip{\nabla(\tau\cdot w_\epsilon)}{\nabla u}}{u_\epsilon}+\frac{2\cdot\tau}{Nu^2_\epsilon}\Big(\Delta u-\frac{|\nabla u|^2}{u_\epsilon}\Big)^2-\frac{|\nabla u|^2}{u^2_\epsilon}+2\frac{\Delta u}{u_\epsilon}+\frac{|\nabla u|^2}{u^2_\epsilon}+\frac{N}{2\tau}\\
&= -2\frac{\ip{\nabla(W_\epsilon+\log u_\epsilon)}{\nabla u}}{u_\epsilon}+\frac{2\cdot\tau}{N}\Big(\frac{\Delta u}{u_\epsilon}-\frac{|\nabla u|^2}{u^2_\epsilon}\Big)^2 +2\frac{\Delta u}{u_\epsilon}+\frac{N}{2\tau}\\
&= -2\frac{\ip{\nabla W_\epsilon}{\nabla u}}{u_\epsilon}+\frac{2\cdot\tau}{N}\Big(\frac{\Delta u}{u_\epsilon}-\frac{|\nabla u|^2}{u^2_\epsilon}\Big)^2 +2\Big(\frac{\Delta u}{u_\epsilon}-\frac{|\nabla u|^2}{u^2_\epsilon}\Big)+\frac{N}{2\tau}\\
&\gs -2\frac{\ip{\nabla W_\epsilon}{\nabla u}}{u_\epsilon}.
\end{split}
\end{equation*}
Together with $\Delta^SW_\epsilon\gs0$, we have proved the lemma.
\end{proof}

Now let us prove Theorem \ref{thm1.3}.

\begin{proof}[Proof of Theorem \ref{thm1.3}] $ \ $\\
(i) The case $N=\infty$. The assumption that $X$ is compact and  $u_0\in L^\infty(X)$ implies that $\mu(X)<\infty$ and $u_0\in L^2(X)$.
Fix an $\epsilon>0$ and set $u_\epsilon=u+\epsilon$. By using Lemma \ref{lem4.1}, we have
\begin{equation*}
\begin{split}
\Delta^*(w_\epsilon u_\epsilon)- \frac{\partial}{\partial t}(w_\epsilon u_\epsilon)\cdot\mu&= u_\epsilon\big(\Delta^*w_\epsilon - \frac{\partial}{\partial t}w_\epsilon \cdot\mu\big)
+2\ip{\nabla w_\epsilon}{\nabla u}\cdot\mu\\
&\gs-2\ip{\nabla w_\epsilon}{\nabla u}\cdot\mu+2\ip{\nabla w_\epsilon}{\nabla u}\cdot\mu=0.
\end{split}
\end{equation*}
Moreover, notice that since $\mu(X)<\infty$, we have that
$|\nabla u|^2\in W^{1,2}\cap L^\infty(X)$ and $u^{-1}_\epsilon\in {\rm Lip}(X)$, and $1\in W^{1,2}(X)$
is an admissible test function for the measure $\Delta^*(|\nabla u|^2/u_\epsilon)$ and $\Delta u$.
From these, we conclude that $\int_X\Delta u=0$ and $\int_X\Delta^*(w_\epsilon u_\epsilon)d\mu=0$, and therefore
$$\int_Xw_\epsilon u_\epsilon d\mu=\int_X(|\nabla u|^2/u_\epsilon-2\Delta u)d\mu= \int_X |\nabla u|^2/u_\epsilon d\mu$$
is monotone decreasing, since $\int_X\frac{\partial}{\partial t}(w_\epsilon u_\epsilon)\,d\mu\le
\int_X\Delta^*(w_\epsilon u_\epsilon)\,d\mu=0$.

{Note that $u(\cdot,t)$ is positive and Lipschitz continuous on $X$
for any $t>0$ and $X$ is compact. Then $f=-\log u$ is also Lipschitz
continuous on $X$ for any $t>0$.} Therefore,
$$\lim_{\epsilon\to0^+}\int_Xw_\epsilon u_\epsilon d\mu=\lim_{\epsilon\to0^+}\int_X|\nabla f|^2\cdot\frac{u^2}{u_\epsilon}d\mu=\int_X|\nabla f|^2ud\mu=\mathcal W_\infty(f,t)$$
 for any $t>0.$ Let $\epsilon\to0^+$, we conclude that $\mathcal W_\infty(f,t)$ is monotone decreasing.

(ii) The case of $N<\infty$. Since $\mu(X)<\infty$, we have $u_0\in L^2(X)$. Fix any $\epsilon>0$. Using Lemma \ref{lem4.2}, we have
\begin{equation*}
\begin{split}
\Delta^*(W_\epsilon u_\epsilon)- \frac{\partial}{\partial t}(W_\epsilon u_\epsilon)\cdot\mu&= u_\epsilon\big(\Delta^*W_\epsilon - \frac{\partial}{\partial t}W_\epsilon \cdot\mu\big)
+2\ip{\nabla W_\epsilon}{\nabla u}\cdot\mu\\
&\gs-2\ip{\nabla W_\epsilon}{\nabla u}\cdot\mu+2\ip{\nabla W_\epsilon}{\nabla u}\cdot\mu=0.
\end{split}
\end{equation*}
This and the same argument as in (i) imply that $\int_XW_\epsilon u_\epsilon d\mu$  is monotone decreasing.
By the definition of $W_\epsilon$, we have
\begin{eqnarray*}
\int_XW_\epsilon u_\epsilon d\mu&&=\int_X\tau\big(\frac{|\nabla u|^2}{u_\epsilon}-\partial_t u\big)
-\Big[\log\big(u_\epsilon\cdot(2\pi\tau)^{N/2}\big)+N\Big]u_\epsilon d\mu\\
&&=\int_X\tau\frac{|\nabla u|^2}{u_\epsilon}-\Big[\log\big(u_\epsilon\cdot(2\pi\tau)^{N/2}\big)+N\Big]u_\epsilon d\mu
\end{eqnarray*}
The same argument in (i) implies that $f=-\log \big(u(4\pi\tau)^{N/2}\big)$ is Lipschitz continuous on $X$ for any $t>0$.
Therefore,
$$\lim_{\epsilon\to0^+}\int_XW_\epsilon u_\epsilon d\mu=\int_X\Big(\tau|\nabla f|^2+f-N\big)ud\mu=\mathcal W_N(f,\tau)$$
 for any $t>0.$ Let $\epsilon\to0^+$, we conclude that $\mathcal W_N(f,t)$ is monotone decreasing. Now we complete the proof.
\end{proof}

\noindent Renjin Jiang

\vspace{0.1cm}
\noindent
School of Mathematical Sciences, Beijing Normal University,
Laboratory of Mathematics and Complex Systems, Ministry of Education, 100875, Beijing, China

\vspace{0.2cm}
\noindent{\it E-mail address}: \texttt{rejiang@bnu.edu.cn}

\vspace{0.2cm}
\noindent Huichun Zhang

\vspace{0.1cm}
\noindent Department of Mathematics, Sun Yat-sen University, Guangzhou 510275, China

\vspace{0.2cm}
\noindent{\it E-mail address}: \texttt{zhanghc3@mail.sysu.edu.cn}

\end{document}